\documentclass{amsart}
\usepackage{amssymb,latexsym, amsmath,amscd,graphics,xspace,ifthen}
\usepackage{amsthm}

\newtheorem{thm}{Theorem}
\newtheorem{lem}{Lemma}

\newcommand{\boxb}{\square_b}
\newcommand{\db}{\overline{\partial}_b}
\newcommand{\dbs}{\overline{\partial}^*_b}
\newcommand{\bt}[1]{{\bar{#1}}}
\newcommand{\Bt}[1]{{\overline{#1}}}
\newcommand{\dom}[1]{\text{Dom}\left(#1\right)}

\begin{document}
\title[Non-compactness of $N_b$.]{Non-compactness of the Neumann operator for the Kohn Laplacian on the Heisenberg ball}

\author{Robert K. Hladky 
}
\address{North Dakota State University Dept. \#2750, 
PO Box 6050, 
Fargo ND 58108-6050}
\email{robert.hladky@ndsu.edu}

\begin{abstract}
For $1\leq q \leq n-2$, we provide explicit examples to demonstrate non-compactness of the Neumann operator for the Kohn Laplacian acting on $L^2$ $(0,q)$-forms on the unit ball in $(2n+1)$-dimensional Heisenberg space.
  \end{abstract}
  
  \maketitle

\section{Introduction}

The purpose of this short note is to fill a missing gap in the literature. In an otherwise excellent paper, Shaw \cite{ShawP} claimed to demonstrate non-compactness for the Neumann operator $N_b$ for the Kohn Laplacian $\boxb$ acting on $L^2$ $(0,1)$-forms on the Heisenberg ball with Neumann boundary conditions. However, the inequality in the second displayed equation of   \cite[p.162]{ShawP} is incorrect and the argument is invalid. Indeed, as we shall see below, the argument cannot be easily fixed.

The Kohn Laplacian is a formally self-adjoint, second order differential operator associated to the tangential Cauchy-Riemann operator $\db$ on Cauchy-Riemann manifolds. It provides a natural sub-elliptic analogue to the elliptic complex Laplacian on complex manifolds. On closed manifolds the qualitative properties, such as regularity and existence of solutions, of both operators are similar \cite{ShawB,FK}. However, on bounded domains the Kohn Laplacian exhibits significantly worse behavior. For domains with sufficiently nice boundaries, the Neumann operator for the complex Laplacian is known to be compact, see for example \cite{Catlin2, Catlin1, ShawB}. Despite the failure of the argument in \cite{ShawP}, the converse has been widely believed for the Kohn Laplacian. In this note we shall provide a new argument to verify this. 

 As with many non-compactness arguments, the idea in \cite{ShawP} was to construct a bounded orthogonal sequence $(\alpha_k)$  such that $(N_b \alpha_k)$ contains no convergent subsequences.  The $\db$-Neumnan problem on bounded domains, and on the Heisenberg ball in particular, has been studied in detail by the author   \cite{Hladky1,Hladky2,Hladky3}.  The idea was to study a degenerate foliation of the Heisenberg ball by spheres and decompose $\boxb$ into pieces transverse and tangential to this foliation. The tangential pieces could be expressed in terms of the Kohn Laplacian on the standard unit sphere and decomposed into eigenspaces.  The sequence constructed in \cite{ShawP} obtained its orthogonality by projecting onto different eigenspaces on the leaves of the foliation. Unfortunately the eigenvalues associated to such a sequence will  converge to infinity and it follows that if $(\alpha_k)$ is bounded then $(N_b \alpha_k)$ must converge to $0$. To construct a sequence contradicting compactness, it is essentially necessary to work within a single eigenspace on the foliating spheres. Thus an entirely different method of constructing the sequence is required.

Here we shall instead show that there is a function $\lambda$ and infinite dimensional subspace of $C^2$ forms with the property that $\boxb u =  \lambda u$. We shall then be able to provide a contradiction to compactness of $N_b$ using sequences from within this subspace.

\section{Definitions and theorem} 
The $2n+1$ dimensional Heisenberg group is the manifold $\mathbb{H}^n =\mathbb{R}_t \times \mathbb{C}^n_z $  equipped with the CR-structure $\mathbb{L}$ defined as the complex linear span of the vector fields $L_j =\partial_{z^j}+i \bt{z}{j} \partial_t$, $j=1,\dots,n$. The Heisenberg  groups are equipped with a anisotropic family of dilations  $\delta_r(t,z)= (r^2t,rz)$ that preserve $\mathbb{L}$ and a dilation-norm $\|(t,z)\| = \left( t^2+|z|^4 \right)^{1/4} $. The bundle $\mathbb{L}$ and its conjugate $\Bt{\mathbb{L}}$ are annihilated by the contact form $\eta = \frac{1}{2} \left( dt -i \bt{z}^j dz^j + i z^j d\bt{z}^j \right)$. 

We define a functions $w \colon \mathbb{H}^n \to \mathbb{C}$ and $\varrho \colon \mathbb{H}^n \to \mathbb{R}$ by $w = t+i|z|^2$ and $\varrho = 1-|w|^2$ respectively. The closed unit ball in $\mathbb{H}^n$ is given by 
\[ \Omega = \{ (t,z) \in \mathbb{H}^n \colon \|(t,z) \| \leq 1 \} =  \{ \left| w \right|^2 \leq 1  \}=\{\varrho \geq 0 \}.\]
 A $(p,q)$-form on $\mathbb{H}^n$ is a differential form of type
 \[ f = \sum f_{a I\Bt{J}}  \eta^a \wedge dz^I \wedge d\bt{z}^J \]
 where $a,I,J$ are multi-indices with $a \in \{ \emptyset , \{1\} \}$, $|a|+|I|=p$ and $|J|=q$. 
The $\db$ operator on $\mathbb{H}^n$ is then defined on smooth $(p,q)$-forms by
\[ \db f  =\sum \left( \Bt{L}_k  f_{aI\Bt{J}} \right) \; d\bt{z}^k \wedge \eta^a \wedge dz^I \wedge d\bt{z}^J.\]
We shall be concerned with $\db$ acting on the space of $L^2$ $(0,q)$-forms on $\Omega$. Accordingly, we extend $\db$ to be the maximal extension of $\db$ to a linear, closed and densely-defined operator $L^2_{(p,q)}(\Omega) \to L^2_{(p,q+1)}(\Omega)$. We denote by $\vartheta_b$ and $\dbs$ the formal adjoint and $L^2$-adjoint of $\db$ respectively.  It is easy to verify from an integration-by-parts argument that a smooth form $u \in \dom{\dbs}$ if and only if
\begin{equation}\label{E:boundary}  \db \varrho \vee u  =0 \quad \text{ on $\partial \Omega$.}\end{equation}
The Kohn Laplacian $\boxb$ on the unit ball is  the unbounded operator $\boxb $ on $ L^2_{(p,q)}(\Omega)$ defined by
\begin{align} \dom {\boxb} &=\{ u \in \dom{\dbs} \cap \dom{\db} \colon \db u \in \dom{\dbs}, \; \dbs u \in \dom{\db} \},\\
\boxb &= \db \dbs + \dbs \db. \end{align}
It was  shown by Shaw in \cite{ShawP} that $\boxb$ is a closed, densely-defined self-adjoint operator. It also follows from the work of \cite{ShawP} that for $p=0$ and $1\leq q \leq n-2$ there is a bounded Neumann operator $N_b \colon L^2_{(0,q)}(\Omega) \to \dom{\boxb} \subset L^2_{(0,q)}(\Omega)$ such that $\boxb N_b =\text{Id}$ on $L^2_{(0,q)}(\Omega)$ and $N_b \boxb = \text{Id}$ on $\dom{\boxb}$.

To describe the boundary conditions for sufficiently smooth forms, we first fix $s=|z|^2$ so $w=t+is$.  Then $\db \rho = iz^k d\bt{z}^k$. Thus if we define $\Bt{Y}:= \bt{z}^k L_\bt{k}$, the boundary condition \eqref{E:boundary} can be written
\begin{equation}\label{E:boundary2}
\Bt{Y} \lrcorner u =0 \quad \text{ on $\partial \Omega$.}
\end{equation}
Now for a multi-index $J$ with $|J|=q+1$, we define a $(0,q)$-form 
\[ \xi^J = \sum\limits_{k\in J} (-1)^{k-1} \bt{z}^k d\bt{z}^{J \backslash k }.\]
A straightforward computation shows that $\Bt{Y} \lrcorner \xi^J =0$ everywhere. 

\begin{lem}\label{L:NonCompact}
Suppose  $1 \leq q \leq n-2$ and $h$ is any holomorphic function on an open set of $\mathbb{C}$ containing the unit ball.  Let  $J=\{1,\dots,q+1\}$  and define a $(0,q)$-form by
\[ u_h :=h(w)  \dfrac{  (z^n)^l}{s^{q+1} }   \xi^J.\]
Then for sufficiently large $l$, $u_h \in \dom{\boxb}$ and 
\[ \boxb u_h = \dfrac{(q+1)(n+l-q-1)}{s} u_h.\]
\end{lem}
 
 \begin{proof}

 Now, as noted earlier, it is easy to see that $\Bt{Y} \lrcorner u_h =0$.  Furthermore 
\[ \db u_h = - (q+1) s^{-1} z^k dz^\bt{k} \wedge u_h + (q+1) h(w)(z^n)^l s^{-(q+1)} d\bt{z}^1 \wedge \dots \wedge  d\bt{z}^{q+1} .\]
As $ \Bt{Y} \lrcorner  d\bt{z}^1 \wedge \dots \wedge  d\bt{z}^{q+1} = \xi^J$, 
we also have $\Bt{Y} \lrcorner \db u_h =0$. Now for very large $l >> 2(q+1)$, $u_h$ is in $C^2(\Bt{\Omega})$ and hence $u_h \in \dom{\boxb}$.  

We can now use the usual formula of Folland and Stein \cite{FollandStein} that on sufficiently smooth $(0,q)$-forms 
\[ \boxb (f_\bt{I} d\bt{z}^I) = -\sum_{k \notin I} (L_k L_\bt{k} f_\bt{I}) \;  d\bt{z}^I - \sum_{k \in I} ( L_\bt{k}  L_k  f_\bt{I}) \;d\bt{z}^I.\]
For the $(0,q)$-form $u_h$ and $1 \leq k \leq q+1$, we have
\[ (u_h)_{\Bt{J \backslash k}}  = (-1)^{k-1} h(w)  \dfrac{  (z^n)^l}{s^{q+1} }  \bt{z}^k .  \]
 By direct computation, we see that for $ j \in J\backslash k$ and $ q+1<m<n$
\begin{align*}
 L_{\bt{j}} L_j   \left( h(w) s^{-a} \right)  &=  i h^\prime(w) s^{-(a+1)} ( s - a z^j \bt{z}^j ) + a h(w) s^{-(a+2)}  ( (a+1) z^j \bt{z}^j-s ),  \\
L_k L_{\bt{k}}  \left( h(w) \bt{z}^k s^{-a}  \right)&=  i h^\prime(w) s^{-(a+1)} \bt{z}^k ( s-a z^k \bt{z}^k )   \\ & \qquad + a h(w) s^{-(a+2) } \bt{z}^k ( -2s  +(a+1) z^k \bt{z}^k),  \\
L_m L_{\bt{m}}  \left( h(w)  s^{-a} \right) &=-ia h^\prime(w) s^{-(a+1)} z^m \bt{z}^m \\
  & \qquad  + a h(w) (z^m)^l s^{-(a+2)}  ( (a+1) z^m \bt{z}^m -s ),\\
  L_n L_{\bt{n}}  \left( h(w) (z^n)^l  s^{-a} \right) &=-ia h^\prime(w) s^{-(a+1)} (z^n)^{l+1} \bt{z}^n \\
  & \qquad  + a h(w) (z^n)^l s^{-(a+2)}  ( (a+1) z^n \bt{z}^n -s (l+1 )). \end{align*}
  From this it is easy to see that
  \begin{align*}  \boxb  &\left(  \frac{ h(w) (z^n)^l }{s^a}  d\bt{z}^{J\backslash k} \right) =  -\frac{i h^\prime(w) (z^n)^l  }{ s^{a+1}}  \left( qs+s-as \right)  \\
  & \qquad -  \frac{a h(w) (z^n)^l}{ s^{a+2}} \left(  (a+1)s  -qs -2s - (n-q-2)s - (l+1)s \right). \end{align*}
The required identity is a simple consequence of the case $a=q+1$.

\end{proof} 

 We can now prove our non-compactness result.
  
 \begin{thm}\label{T:NonCompact}
 If $1 \leq q \leq n-2$, the Neumann operator $N_b$ on $L^2_{(0,q)} (\Omega)$ is not compact.
 \end{thm}
 
 \begin{proof}
 
 First fix $l>>0$ and  note that  $\langle u_{h_1}, u_{h_2} \rangle =  \frac{ h_1 \Bt{h_2} }{s^{2q+2} }  |z^n|^{2l} \sum\limits_{k=1}^{q+1} \left| z^k \right|^2$. Let  $D = \{ \zeta \in \mathbb{C} \colon \left| \zeta \right| \leq 1, \text{Im}(\zeta) \geq 0\}$ and set $\omega = \int_{\mathbb{S}^{2n-1}}  |z^n|^{2l} \sum\limits_{k=1}^{q+1} z^k \bt{z}^k d\sigma$ where $\mathbb{S}^{2n-1}$ is the unit ball in $\mathbb{C}^{2n}$. Then if we switch to polar coordinates (and recall that $s=|z|^2$ rather than $|z|$), we see that, for some constant $C>0$ independent of $h$,  \begin{equation}\label{E:norm} \| u_h \|^2_{L_{(0,q)} ^2(\Omega)} = C \omega \int_D  \left| h(\zeta)\right|^2 s^{n+l-2q-1} dV_D \end{equation}
 where we also use $s$ to denote $ \text{Im}(\zeta)$. For convenience, we renormalize  the norm on $L^2_{(0,q)}$ so that $C\omega =1$.
  
Now choose $h_k(\zeta) =\sqrt{8k+2} \zeta^{4k}$ and let $K=\{ re^{i \theta} \colon 1/2 < r <1, \pi/4 < \theta < 3\pi/4\}$.  Then it is easy to see that  for $k >j>0$, $h_j$ and $h_k$ are orthogonal on $L^2(K)$ and so
\[ \int_K \left| h_k - h_j\right|^2 dV_D  = \int_K \left|h_k\right|^2 + \left|h_j\right|^2 dV_D\geq \frac{\pi}{4} + \frac{\pi}{4} =\frac{\pi}{2} . \]
From this it follows that
\begin{equation}\label{E:norm2}
\begin{split}
 \| u_{h_k} - u_{h_j} \|^2_{L_{(0,q)} ^2(\Omega)} & = \| u_{(h_k-h_j)} \|^2_{L^2(\Omega)}  = \int_D \left| h_k -h_j \right|^2 s^{n+l-2q-1} dV_D \\
 & \geq \int_K \left| h_k -h_j \right|^2 s^{n+l-2q-1} dV_D \\
 &  \geq  \left( \frac{\sqrt{2}}{4} \right)^{n+l-2q-1} \frac{\pi}{2}. 
 \end{split}
 \end{equation}
Furthermore, since $l$ has been chosen to be very large, $s^{n+l-2q-3} \leq 1$ on $D$. Thus
\begin{equation}\label{E:norm3} \left\| s^{-1} u_{h} \right\|_{L^2_{(0,q)}(\Omega)}^2 \leq \int_D \left| h(\zeta) \right|^2 s^{n+l-2q-3} dV_D \leq \int_D \left| h(\zeta) \right|^2 \; dV_D \leq 2\pi. \end{equation}
Now $u_{h_k}$ is a  sequence of forms in $C_{(0,q)} ^2(\Bt{\Omega}) \cap \dom{\boxb}$. By   \eqref{E:norm2}  there are no convergent subsequences in $L_{(0,q)} ^2(\Omega)$.  By Lemma \ref{L:NonCompact} and \eqref{E:norm3}, the sequence $\boxb u_{h_k}$ is bounded in $L_{(0,q)} ^2(\Omega)$. It follows trivially that $N_b$ is not compact.
  
 \end{proof}
 
 One consequence of this theorem is that the range of $N_b$ cannot be contained in any positive Sobolev space $H^\epsilon(\Omega)$ as $N_b$ would need to be compact by the Rellich Lemma. Thus solutions to $\boxb u =f$ do not globally gain in Sobolev regularity on the Heisenberg ball. Gains in weighted regularity are described in \cite{Hladky1,Hladky2,Hladky3}.
  
 As a final remark, we note that the Heisenberg ball has two characteristic boundary points, $(\pm 1, 0)$ where the distribution $\mathbb{L}$ is tangent to the boundary.  The presence of characteristic points typically complicates any analysis of $\db$ or $\boxb$ on the domain. However, this  is not a contributing factor to Theorem \ref{T:NonCompact}. Indeed, an almost identical argument could be applied to the domain $\{ | w-2i| < 1\}$ which has completely non-characteristic boundary.

 \end{document}